\documentclass[letterpaper, 10 pt, conference]{ieeeconf}  

\IEEEoverridecommandlockouts                              

\usepackage{stfloats}
\usepackage{savesym}
\usepackage{amsmath}
\usepackage{textcomp}
\savesymbol{iint}
\restoresymbol{TXF}{iint}
\usepackage{graphicx,subfigure}
\usepackage{epstopdf}
\usepackage{amssymb,amsfonts}
\usepackage{cite}
\usepackage{latexsym}
\usepackage{psfrag}
\usepackage{pifont}
\usepackage{caption}
\usepackage{color}
\usepackage{tikz}
\usepackage{bm}
\usetikzlibrary{arrows,shapes,chains}
\usepackage{bbm,mathrsfs}%
\usepackage{fancyhdr} 
\usepackage{verbatim} 
\usepackage{xcolor}

\usepackage{setspace}
\usepackage{algorithm}
\usepackage{algpseudocode}

\let\labelindent\relax
\usepackage{enumitem}
\newtheorem{theorem}{Theorem}   
\newtheorem{lemma}{Lemma}

\newtheorem{example}{Example}
\newtheorem{remark}{Remark}

\newtheorem{definition}{Definition}

\title{\LARGE \bf
Fault Diagnosis and Prognosis in Partially-Observed Discrete Event Systems with Delayed Observations}

\author{~Jiwei Wang,~Simone Baldi,~Wenwu Yu,~Xiang Yin
	\thanks{This work was supported in part by the Jiangsu Provincial Scientific Research Center of Applied Mathematics grant BK20233002, and in part by the National Natural Science Foundation of China grants 62150610499, 62073074, 62073076, 62233004, 62173226.}
	\thanks{J. Wang is with School of Cyber Science and Engineering, Southeast University, Nanjing 210096, China (e-mails: jwwang@seu.edu.cn).}
	\thanks{S. Baldi and W. Yu are with School of Mathematics, Southeast University, Nanjing 210096, China (e-mails: \{simonebaldi,wwyu\}@seu.edu.cn).}
	\thanks{X. Yin is with the Department of Automation and Key Laboratory of System Control and Information Processing, Shanghai Jiao Tong University, Shanghai 200240, China (e-mail: yinxiang@sjtu.edu.cn).}
}

\begin{document}

\maketitle
\thispagestyle{empty}
\pagestyle{empty}

\begin{abstract}

Fault diagnosis and prognosis in discrete event systems are studied in the scenario where the observations are possibly received with delay.
To address this scenario, two conditions for diagnosis and prognosis with delayed observations are proposed, where we show that the state-of-the-art notion of prognosability must be revised to avoid conservativeness.
Diagnosability and prognosability conditions are then verified by introducing a delay observer and a new verification function. 
Theoretical analysis indicates the effectiveness of the verification method for fault diagnosis and prognosis in the system.
\end{abstract}

\section{INTRODUCTION}

Fault analysis is fundamental for system reliability and maintenance.
Discrete event systems (DES) \cite{cassandras2009introduction} provide a useful framework for fault analysis: structured DES approaches to fault analysis have been proposed using observations of the system to detect (fault diagnosis) or predict (fault prognosis) anomalies and deviations from healthy operation.
Notions of diagnosability \cite{sampath1995diagnosability} and prognosability \cite{genc2009predictability} have been presented for DES, where the former determines the ability to timely detect faults, while the latter determines the ability to forecast future faults based on current system conditions.
Prognosability is also called predictability in some literature.

With DES encompassing applications across various domains \cite{lin1994diagnosability}, \cite{he2024asynchronous}, the literature has proposed several settings for DES fault diagnosis and prognosis, spanning decentralized scenarios \cite{debouk2000coordinated,kumar2009decentralized}, robustness notions \cite{yin2019robust,liao2022robust}, fuzzy systems\cite{liu2009diagnosability,benmessahel2017predictability}, fault tolerant control \cite{paoli2011active,watanabe2017safe}, among others. 
The interested reader can refer to \cite{lafortune2018history,watanabe2021fault} for more details.
A fundamental setting in fault analysis is handling observations received with delay.
Such issues have been addressed in network systems \cite{nunes2018codiagnosability,liu2021online,viana2021codiagnosability}, and distributed scenarios with transmission delay \cite{qiu2008distributed,wang2024distributed,takai2011distributed}, introducing notions of codiagnosability and coprognosability.
Yet, these results do not encompass the delayed-observation scenario considered in this work. 
This point will be discussed making use of a suitable example (cf. Example 2\&3 in this manuscript).

To consider the delayed-observation scenario, one needs to extend state-of-the-art notions of diagnosability in $K$ steps, such as $K$-diagnosability \cite{cassez2008fault}. 
We do this by introducing a new $ K^T $-diagnosability, where $T$ is a delay-dependent parameter.
We show that the state-of-the-art notion of prognosability \cite{genc2009predictability} cannot be extended to the delayed-observation scenario. 
We thus propose a notion of prognosability that, while being equivalent to the state-of-the-art prognosability in the absence of delays, it can be extended to the delayed-observation scenario, resulting in $ T $-constrained prognosability.
We finally propose a delay observer structure to record the delays required for diagnosis and prognosis, which in turn leads to a new function to verify $ K^T $-diagnosability and $ T $-constrained prognosability.

The remainder of this paper is organized as follows.
In Section \ref{2}, the partially-observed DES is introduced.
Section \ref{3} recalls the state-of-the-art notions for diagnosability and prognosability, and proposes novel definitions of $ K^T $-diagnosability and $ T $-constrained prognosability.
In Section \ref{4}, we verify such conditions by introducing the delay observer and the verification function.
Section \ref{8} concludes this paper.

\section{Preliminaries}\label{2}

Let us start by considering a finite event set $E$ as an alphabet, enabling us to interpret the \emph{concatenation} of word strings within the alphabet as finite sequences of events in $E$.
Let $E^*$ denote the set of all finite strings over $E$. 
The length of a string $s\in E^*$ is denoted as $|s|$, and let $\epsilon$ represent the empty string with $|\epsilon| = 0$. 
A \emph{language} $L$ constitutes a collection of event strings, derived from events in $E$.
The \emph{prefix-closure} of a language $L$ is defined as $\overline{L}=\{s\in E^* | \exists t\in E^*, \text{ s.t. } st\in L\}$, and $L$ is \emph{prefix-closed} if $L=\overline{L}$.
The \emph{post-language} of $L$ following a string $s$ is denoted as $L\backslash s = \{s'\in E^*\mid ss'\in L\}$.
We say that a language $L$ is \emph{live} if, for every $s \in L$, there exists an $e \in E$ such that $se\in L$, representing that any string in $L$ can be extended to any length.

Automata are a common framework for manipulating languages, used to model DES.
Let us consider the finite automaton
\begin{equation}\label{S}
G=(X, E, \alpha, X_0),
\end{equation}
where $X$ is the set of finite states, $E$ is the set of finite events, $ \alpha:X \times E^* \rightarrow 2^X $ is the transition function describing the transition of an event string, $X_0 \subseteq X $ is the set of possible initial states.
The language generated by $G$ from state $x \in X$ is denoted by $\mathcal{L}(x,G)=\{s \in E^* | \alpha (x,s)! \}$, where $!$ indicates that the string $ s $ ``is defined'', meaning that it can occur starting from state $x$.
If $x_0 \in X_0$, we simply denote $\alpha(x_0,s)$ as $\alpha(s)$ and $\mathcal{L}(x_0,G)$ as $\mathcal{L}(G)$. 

In practice, the occurrence of certain events may not be detectable.
Accordingly, we partition $ E $ into the observable events $ E_o $ and the unobservable events $ E_{uo} $.
Let us now introduce two projection operators.
The operator $ P_{E_o} $ is utilized for acquiring the observable events in an event string:
$\forall s\in \mathcal{L}(G), \forall e\in E: \alpha(se)!, $
\begin{equation}
P_{E_o}(\epsilon)= \epsilon, P_{E_o}(se)=\left\{
\begin{aligned}
& P_{E_o}(s)e, &\text{if} \ e\in E_o; \\
& P_{E_o}(s),  &\text{if} \ e\notin E_o.
\end{aligned}
\right.
\end{equation}
Intuitively, for any system trajectory $ s\in \mathcal{L}(G) $, $ P_{E_o}(s) $ gives the observed events.
We extend $ P_{E_o} $ to handle a set of event string, that is, $ \forall S \subseteq \mathcal{L}(G) $, $ P_{E_o}(S)=\{ s\in E_o^*| \exists s' \in S, \text{ s.t. } s=P_{E_o}(s') \} $.
We then introduce the operator $ \zeta^n_{E_o} $: 
$ \forall s \in \mathcal{L}(G) $,
\begin{equation*}
\zeta^n_{E_o}(s) \!=\! \{ s'' \!\in\! \overline{s} | \exists s' \!\in\! \overline{s}: |s'| \!\geq\! |s|-n, \text{s.t.} P_{E_o}(s'') \!=\! P_{E_o}(s') \},
\end{equation*}
where $ n\in\mathbb{N} $ is the number of steps.
Intuitively, $ \zeta^n_{E_o} $ collects all prefixes of $ s $ which have the same observation as a system trajectory $ s'\in\overline{s} $ under the observation ability $ E_o $.
Obviously, given $ n_1\geq n_2 $, we have $ \zeta^{n_2}_{E_o}(s)\subseteq\zeta^{n_1}_{E_o}(s) $.

\subsection{Illustrative example}
\label{IE}
Throughout this paper, we take an air heating unit example to illustrate the key concepts.
The example describes a start-up process of the unit via the automaton $ G $ in Fig. \ref{System}, where the system is off initially (in state $ x_0 = \{0\} $).
In healthy conditions, after turning the air heating unit on, the air flow sensor 2 observes that the flow rate is regular ($ e_2, 0 \to 1 $: $ e_2 $ occurs and the system state goes from $ 0 $ to $ 1 $); after some time, the temperature sensor 1 observes that the desired temperature is reached ($ e_1, 1 \to 2 $).
Then, sensor 2 keeps monitoring the flow rate ($ e_2, 2 \to 2 $).
However, if the air quality sensor 3 observes an abnormal amount of dust (e.g., possibly due to aging or some leak in the duct) ($ e_3, 0 \to 3 $), after some unobservable anomalies (e.g., unusual vibration of the fan) ($ u, 3 \to 4 $), the fan may fail or clog ($ f, 4 \to 5 $).
As a result, sensor 1 observes high temperature ($ e_1, 5 \to 6 $), leading to overheating of the coil with no flow.
This damage may be irreversible even in the scenario that the fan recovers from the clog and restores the flow rate ($ e_2, 6 \to 6 $).
	
\begin{example}\label{ES}(System model).
	Using the automaton formalism in \eqref{S}, we have \begin{equation*}
	G=(\{0,1,2,3,4,5,6\}, \{e_1,e_2,e_3,u,f\}, \alpha, \{0\}).
	\end{equation*}
	When the system is in state $ 5 $, the only event that can occur is $ e_1 $, that is, $ \alpha(5,e_1)! $. 
	Let $ E_{o} = \{e_1, e_2, e_3\} $ be the observable events.
	Then, for the string $ e_3ufe_1e_2 $ generated by $ G $, we have
	\begin{align*}
	P_{E_o}(e_3ufe_1e_2) &= e_3e_1e_2,\\ 
	\zeta^0_{E_o}(e_3ufe_1e_2) &= \{e_3ufe_1e_2\},\\
	\zeta^2_{E_o}(e_3ufe_1e_2) &= \{e_3,e_3u,e_3uf,e_3ufe_1,e_3ufe_1e_2\}.
	\end{align*}
\end{example}

\begin{figure}[htbp]
	\centering
	\subfigure[Air heating unit]
	{\label{thermostat}
		\includegraphics[height=1.7cm]{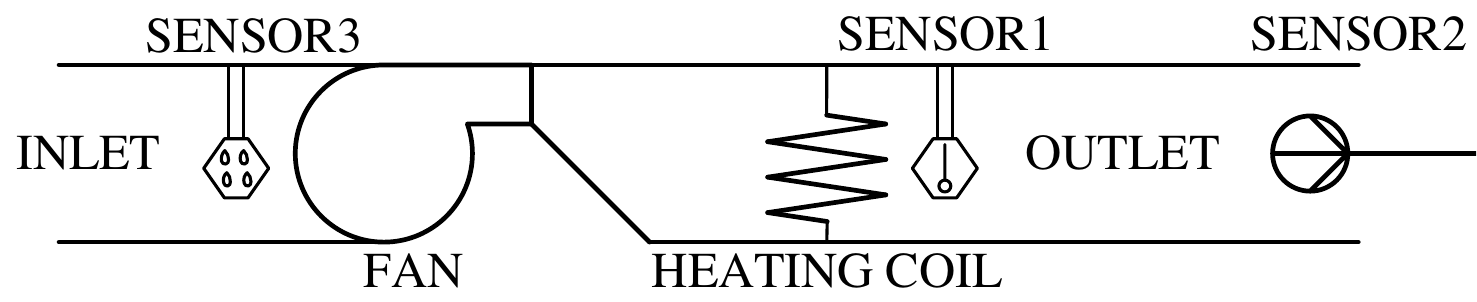}}
	\subfigure[Model $ G $ of Air Heating Unit]
	{\label{System}
		\includegraphics[height=1.6cm]{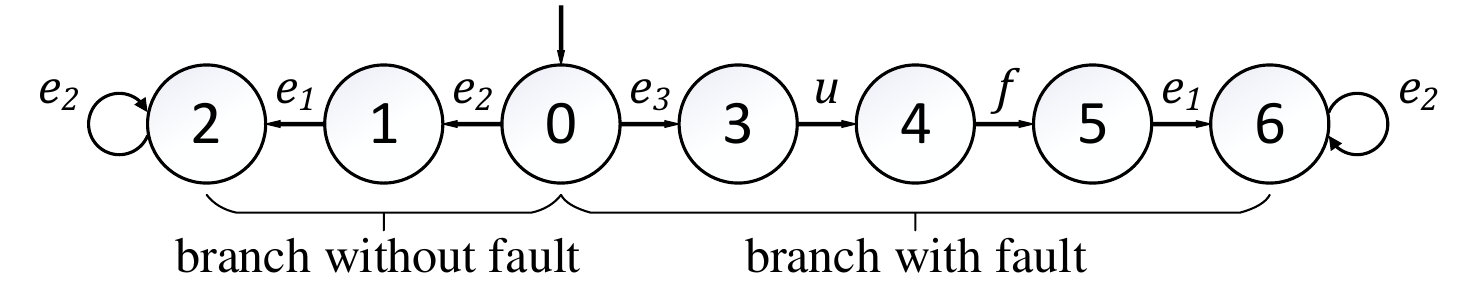}}
	\caption{Air heating unit and the model $ G $ of its start-up process.}
\end{figure}

\section{Diagnosability and Prognosability with Delayed Observation}\label{3}

In this section, we first recall and then extend state-of-the-art notions of $ K $-diagnosability and prognosability.

\subsection{Diagnosis with Delay}
Let $K$ be the maximum number of steps permitted from the occurrence of a fault to its diagnosis.
Let $G=(X, E, \alpha, x_0)$ be the system model.
We introduce a structure of step counter $ \Delta: \mathcal{L}(G) \to \{-1, 0, 1, \dots, K\} $ to count the number of steps in an event string after a fault event $f$ occurs: $ \forall s \in \mathcal{L}(G), \forall e\in E: \alpha(se)! \Rightarrow \Delta(\epsilon)=-1, \Delta(se)= $
\begin{align}\label{Delta}
\left\{
\begin{aligned}
& \Delta(s), &\text{if}\ & [\Delta(s) \!=\! -1 \wedge e \!\neq\! f ]\vee [\Delta(s) \!=\! K]; \\
& \Delta(s) \!+\! 1, &\text{if}\ &  [\Delta(s) \!=\! -1 \wedge e \!=\! f ]\vee [0 \!\leq\! \Delta(s) \!<\! K];
\end{aligned}
\right.
\end{align}
where $ -1 $ means no fault happens.
By means of $ \Delta $, let us recall the notion of $K$-diagnosability.
\begin{definition}\label{KD}
	($ K $-diagnosability\cite{cassez2008fault})  
	For $ K \in \mathbb{N} $, a live language $ \mathcal{L}(G) $ is $ K $-diagnosable w.r.t. $ f $ if $ \forall s \in \mathcal{L}(G): \Delta(s) = K $,
	\begin{equation}\label{KDC}
	\forall s'\in \mathcal{L}(G): P_{E_o}(s')=P_{E_o}(s) , \Delta(s') \neq -1.
	\end{equation}
\end{definition}
\vspace{2mm}

Intuitively, \eqref{KDC} means that any string with the same observation as $ s $ must have the fault event $ f $ in it, which implies that the fault can be always diagnosed after it occurs within $ K $ steps.

Now we consider the scenario that the observed signals are processed and transferred with an unknown delay.
We denote $ D\geq0 $ as the transmission distance, and $ T>0 $ as a coefficient related to transmission efficiency.
If the observation site observes an event, then it will be received with a delay of no more than $ \lceil \frac{D}{T} \rceil $ steps ($ \lceil \cdot \rceil $ rounds the element to the nearest integer towards infinity).
We give the following definition for fault diagnosis with delayed observation.
\begin{definition}\label{KTCD}
	($ K^{T} $-diagnosability) For $ K \in \mathbb{N} $, $ T > 0 $ and $ D\geq0 $, a live language $ \mathcal{L}(G) $ is $ K^{T} $-diagnosable w.r.t. $ f $ and $ D $ if $ \forall s \in \mathcal{L}(G): \Delta(s) = K, $
	\begin{equation}\label{KTCDC}
	\forall s'\!\in\! \mathcal{L}(G)\!:\! P_{E_o}(\zeta_{E_o}^{\lceil \frac{D}{T} \rceil}(s')) \cap P_{E_o}(\zeta_{E_o}^{\lceil \frac{D}{T} \rceil}(s)) \!\neq\! \emptyset,	\Delta(s') \!\neq\! -1.
	\end{equation}
\end{definition}

\begin{remark}
	By comparing \eqref{KDC} with (\ref{KTCDC}), the presence of delay enlarges the state set that the system may be in.
	If (\ref{KTCDC}) holds, all key events to determine the occurrence of $ f $ can be timely received.
	Intuitively, $ K^{T} $-diagnosability allows that any faulty string ($ s $ satisfying $ \Delta(s) = K $) can be distinguished from a healthy string ($ s $ satisfying $ \Delta(s) = -1 $), despite the delay. 
	As $ T $ increases, the strings that cannot be distinguished from the string $ s: \Delta(s) = K $ become less and less, that is, less strings would satisfy (\ref{KTCDC}).
	As expected, $ K^{T'} $-diagnosability implies $ K^{T} $-diagnosability when $ T' \leq T $ (higher transmission efficiency improves diagnosis ability).
	Compared to existing work on diagnosability subject to observation delays \cite{viana2021codiagnosability}, Definition \ref{KTCD} focuses on determining whether key events for diagnosis can be received in a timely manner.
	\hfill \ensuremath{\Box}
\end{remark}

\begin{example}\label{EKTD}($ K^{T} $-diagnosability in the presence of delay).
	Consider the air hearing unit in Sect. \ref{IE}.
	Delayed observations may happen when the air heating unit is part of a central energy management system where all observed signals are processed and transferred.
	Suppose $ D=3 $ and $ K=3 $.
	For $ T=1 $, the occurrence of the events in $ E_o $ will be received in no more than $ \lceil \frac{D}{T} \rceil=3 $ steps.
	Then, for the string $ s=e_3ufe_1e_2e_2 $ generated by $ G $, we have $ \Delta(s) = 3 $ and $ P_{E_o}(\zeta_{E_o}^{3}(s)) = \{e_3,e_3e_1,e_3e_1e_2,e_3e_1e_2e_2\} $.
	Since $ P_{E_o}(\zeta_{E_o}^{3}(s)) \cap P_{E_o}(\zeta_{E_o}^{3}(e_3))= \{e_3\} $ and $ \Delta(e_3)=-1 $, we have that $ \mathcal{L}(G) $ is not $ K^{T} $-diagnosable w.r.t. $ f $ and $ D $ for $ K=3,T=1,D=3 $.
	However, for $ T=2 $, we have $ \lceil \frac{D}{T} \rceil=2 $ steps and $ P_{E_o}(\zeta_{E_o}^{2}(s)) = \{e_3e_1,e_3e_1e_2,e_3e_1e_2e_2\} $.
	For any string $ s' $ with $ \Delta(s')=-1 $, we have
	$ P_{E_o}(\zeta_{E_o}^{2}(s'))\cap P_{E_o}(\zeta_{E_o}^{2}(s)) = \emptyset $, implying that $ \mathcal{L}(G) $ is $ K^{T} $-diagnosable w.r.t. $ f $ and $ D $ for $ K=3,T=2,D=3 $.
	\hfill \ensuremath{\Box}
\end{example}

\subsection{Prognosis with Delay}

Prognosability allows faults to be prognosed before their occurrence.
The literature has introduced the following notion of prognosability.
\begin{definition} \label{Dprog}
	(Prognosability \cite{genc2009predictability}). A live language $ \mathcal{L}(G) $ is prognosable w.r.t. $ f $ if $ \forall sf \in \mathcal{L}(G): s\in E^* $, 
	\begin{equation}\label{PC}
	\begin{aligned}
	&\exists s' \in \bar{s}: \Delta(s')=-1 , \text{s.t. } \forall t\in \mathcal{L}(G):  \\
	&\Delta(t)=-1 \wedge P_{E_o}(t) = P_{E_o}(s'), \exists n\in\mathbb{N}, \text{s.t.} \\ 
	&\forall t'\in\mathcal{L}(G)\backslash t, [ |t'|\geq n \Rightarrow \Delta(tt')\geq0 ].
	\end{aligned}
	\end{equation}
\end{definition}
\vspace{2mm}

Intuitively, \eqref{PC} means that any string that has the same observation as $ s'\in \bar{s} $ will always lead to the occurrence of $ f $, implying that we definitely know that the fault will happen after observing $ P_{E_o}(s') $.
Later we will show with an example that Definition \ref{Dprog} cannot handle delay well. 
We now introduce a different definition of prognosability that, while being equivalent to Definition \ref{Dprog} in the absence of delays, is extendable in the scenario of delayed observation.
\begin{definition} \label{DprogR}
	(Prognosability). A live language $ \mathcal{L}(G) $ is prognosable w.r.t. $ f $ if $ \forall s \in \mathcal{L}(G): \Delta(s)=0 $,
	\begin{equation}\label{PCR}
		\begin{aligned}
		&\exists s' \in \bar{s}, \text{s.t.} \forall t\in \mathcal{L}(G): \Delta(t)=-1 \wedge P_{E_o}(t) = P_{E_o}(s'),\\
		&\exists n\in\mathbb{N}, \text{s.t.} \forall t'\in\mathcal{L}(G)\backslash t, [ |t'|\geq n \Rightarrow \Delta(tt')\geq0 ].
		\end{aligned}
	\end{equation}
\end{definition}

The following lemma shows that Definition \ref{DprogR} is equivalent to Definition \ref{Dprog}.
\begin{lemma}
	Let $G$ in (\ref{S}) be the system model, $f$ be the fault events.
	Then, for any $ sf \in \mathcal{L}(G) $, \eqref{PC} holds if and only if for any $ s \in \mathcal{L}(G): \Delta(s)=0 $, \eqref{PCR} holds.
\end{lemma}
\begin{proof}
	As necessity is obvious, let us only consider sufficiency.
	Suppose that for any $ s \in \mathcal{L}(G): \Delta(s)=0 $, \eqref{PCR} holds.
	We first show that for any $ s \in \mathcal{L}(G): \Delta(s)=0 $, \eqref{PC} holds.
	Consider $ s' $ in \eqref{PCR}.
	If $ \Delta(s')=-1 $, then \eqref{PC} holds directly.
	If $ \Delta(s')\neq-1 $, then there exists $ s_1\in \overline{s'}\subseteq\bar{s} $ such that $ \Delta(s_1f)=0 $, which implies that
	\begin{equation}\label{Lp}
	\begin{aligned}
	&\exists s_2\in \overline{s_1f}, \text{s.t.} \forall t\in \mathcal{L}(G): \Delta(t)=-1 \wedge P_{E_o}(t) = P_{E_o}(s_2),\\
	&\exists n\in\mathbb{N}, \text{s.t.} \forall t'\in\mathcal{L}(G)\backslash t, [ |t'|\geq n \Rightarrow \Delta(tt')\geq0 ].
	\end{aligned}
	\end{equation}
	Obviously, the only case that $ \Delta(s_2)\neq-1 $ is $ s_2=s_1f $.
	Because $ P_{E_o}(s_1f) = P_{E_o}(s_1) $, we can set $ s_2=s_1 $ to make \eqref{Lp} hold.
	Since $ s_2\in \overline{s_1f}\subseteq\bar{s} $, we conclude that for any $ s \in \mathcal{L}(G): \Delta(s)=0 $, \eqref{PC} holds.
	We then show that for any $ sf \in \mathcal{L}(G): s\in E^* $, \eqref{PC} holds.
	If $ \Delta(s)=-1 $, then $ \Delta(sf)=0 $, implying that \eqref{PC} holds.
	If $ \Delta(s)\neq-1 $, then there exists $ s_1\in \bar{s} $ such that $ \Delta(s_1f)=0 $, which also implies that \eqref{PC} holds.
\end{proof}

Now we extend Definition \ref{DprogR} in the scenario of delayed observation, when the message from the observation site can be received in no more than $ \lceil \frac{D}{T} \rceil $ steps.
\begin{definition}\label{DTP}
	($ T $-constrained prognosability). For $ T > 0 $ and $ D\geq0 $, a live language $ \mathcal{L}(G) $ is $ T $-constrained prognosable w.r.t. $ f $ and $ D $ if $ \forall s \in \mathcal{L}(G): \Delta(s)=0 $,
	\begin{equation}\label{TPC}
	\begin{aligned}
	&\exists s' \in \bar{s}, \text{ s.t. } \forall t\in \mathcal{L}(G): \Delta(t)=-1 \wedge \\
	&P_{E_o}(\zeta_{E_o}^{\lceil \frac{D}{T} \rceil}(t))\cap P_{E_o}(\zeta_{E_o}^{\lceil \frac{D}{T} \rceil}(s'))\neq\emptyset, \exists n\in\mathbb{N}, \\ 
	&\text{s.t. } \forall t'\in\mathcal{L}(G)\backslash t, [ |t'|\geq n \Rightarrow \Delta(tt')\geq0 ].
	\end{aligned}
	\end{equation}
\end{definition}\vspace{2mm}

\begin{remark}
	Due to the transmission delay, receiving the occurrence of an event before the fault $ f $ does not mean that $ f $ has not occurred yet.
	Hence, condition \eqref{TPC} means that any string that may lead to the same received observation as $ s'\in \bar{s} $ will always lead to the occurrence of $ f $.
	If (\ref{TPC}) holds, all key events to determine if $ f $ will occur can be timely received.
	In other words, we definitely know that $f$ will occur before its occurrence despite the delay. 
	Note that $ T' $-constrained prognosability implies $ T $-constrained prognosability when $ T' \leq T $.
	Compared to existing work on prognosability subject to observation delays \cite{takai2011distributed}, Definition \ref{DTP} provides a less conservative definition, as illustrated below.
\end{remark}


\begin{example}\label{ETDP}($ T $-constrained prognosability in the presence of delay).
	Consider the system $ G $ in Fig. \ref{System} and the string $ s=e_3uf $ generated by $ G $.
	Suppose $ D=3 $ and $ T=2 $.
	Then, we know that $ e_3 $ will be received within $ \lceil \frac{D}{T} \rceil=2 $ steps before the occurrence of $ f $, that is, $ f $ can be prognosed.
	Given $ P_{E_o}(\zeta_{E_o}^{2}(e_3uf)) = \{e_3\} $, we have that the set of strings $ t $ satisfying $ \Delta(t)=-1 \wedge P_{E_o}(\zeta_{E_o}^2(t))\cap \{e_3\}\neq\emptyset $ is $ \{e_3,e_3u\} $, which will lead to $ f $ in 2 steps.
	Hence, we conclude that $ \mathcal{L}(G) $ is $ T $-constrained prognosable w.r.t. $ f $ and $ D $ for $ T=2 $ and $ D=3 $.
	Let us now consider the extension from Definition \ref{Dprog}, which is equivalent to \cite[Definition 3]{takai2011distributed} in the centralized case: 
	$ \forall sf \in \mathcal{L}(G): s\in E^* $, 
	\begin{equation}\label{0}
	\begin{aligned}
	&\exists s' \in \bar{s}: \Delta(s')=-1 , \text{s.t. } \forall t\in \mathcal{L}(G):  \\
	&\Delta(t)=-1 \wedge P_{E_o}(\zeta_{E_o}^{\lceil \frac{D}{T} \rceil}(t))\cap P_{E_o}(\zeta_{E_o}^{\lceil \frac{D}{T} \rceil}(s'))\neq\emptyset, \\ 
	& \exists n\in\mathbb{N}, \text{s.t.} \forall t'\in\mathcal{L}(G)\backslash t, [ |t'|\geq n \Rightarrow \Delta(tt')\geq0 ].
	\end{aligned}
	\end{equation}
	For the string $ s=e_3uf $, we have that the set of strings $ s'\in \overline{s}: \Delta(s')=-1 $ is $ \{\epsilon,e_3,e_3u\} $ and $ \forall s'\in \{\epsilon,e_3,e_3u\}, \epsilon\in P_{E_o}(\zeta_{E_o}^2(s')) $.
	However, if we consider the string $ e_2 $, we have $ \epsilon\in P_{E_o}(\zeta_{E_o}^2(e_2)) $ and \begin{equation*}
	\forall n\in\mathbb{N}, \exists t'\in\mathcal{L}(G)\backslash e_2:|t'|\geq n, \text{s.t.} \Delta(e_2t')=-1,
	\end{equation*}
	implying that $ s $ does not satisfy \eqref{0}.
	In other words, the extension \eqref{0} of Definition \ref{Dprog} fails to capture that $ f $ can be prognosed. 
	This justifies the introduction of Definition \ref{DprogR} and its extension in Definition \ref{DTP}.
	\hfill \ensuremath{\Box}
\end{example}

\section{Verification of Diagnosability and Prognosability} \label{4}

In general, it is intractable to determine if a fault can be diagnosed or prognosed by analysing each event string. 
It is necessary to embed the delay information into the automaton and develop a feasible method to verify $ K^T $-diagnosability and $ T $-constrained prognosability.
This is done by linking the system states to fault events and by building a delay observer structure to handle the delays.

\subsection{Delay Observer}\label{SDO}

A delay observer aims to register the delays of the events that help to distinguish the faulty from the healthy strings.
To construct a delay observer, a fault automaton structure \cite{yin2015codiagnosability} is utilized to track the number of steps following a fault occurrence
\begin{equation}\label{FA}
\hat{G}=(\hat{X}, E, \hat{\alpha}, \hat{x}_{0}),
\end{equation}
where $ \hat{X} = X\times \{-1, 0, 1, \dots, K\} $ comprises the state in $ X $ of $ G $ along with the fault counting component, denoted as $ \hat{x} = (x,|\hat{x}|_f) \in \hat{X} $ where $ |\hat{x}|_f $ represents the number of steps following a fault occurrence.
The transition function $\hat{\alpha}: \hat{X} \times E \to \hat{X} $ is specified as follows: $ \forall \hat{x}=(x,|\hat{x}|_f)\in \hat{X}, e\in E:\alpha(x,e)!, $
\begin{equation*}
 \hat{\alpha}((x,|\hat{x}|_f),e) = (\alpha(x,e),|\hat{x}|_f+v),
\end{equation*}
where $ |\hat{x}|_f \in \{-1, 0, 1, \dots, K\} $ and $ v $ is determined by
\begin{equation*}
v = \left\{
\begin{aligned}
& 0 ,\quad \text{if} \ [|\hat{x}|_f = -1 \wedge e \neq f ]\vee [|\hat{x}|_f = K];\\
& 1 ,\quad \text{if} \ [|\hat{x}|_f = -1 \wedge e = f ]\vee [0 \leq |\hat{x}|_f < K].
\end{aligned}
\right.
\end{equation*}
The initial state is $ \hat{x}_{0} =(x_0, -1)$.
Note that $ \mathcal{L}(\hat{G}) = \mathcal{L}(G) $.
In both diagnosis and prognosis problems, we aim to distinguish two set of states, denoted by $ \hat{X}^F $ (faulty) and $ \hat{X}^H $ (healthy).
In the diagnosis problem, we distinguish 
\begin{equation}\label{diag}
\begin{aligned}
\hat{X}^F&=\{\hat{x}\in \hat{X}\mid |\hat{x}|_f=K\},\\
\hat{X}^H&=\{\hat{x}\in \hat{X}\mid |\hat{x}|_f=-1\},
\end{aligned}
\end{equation}
while in the prognosis problem, we distinguish
\begin{equation}\label{prog}
\begin{aligned}
\hat{X}^{F}=\{&\hat{x}\in \hat{X}\mid |\hat{x}|_f=0\},\\
\hat{X}^H=\{&\hat{x}\in \hat{X} \mid \forall n\in \mathbb{N}, \exists s\in\mathcal{L}(\hat{x},\hat{G}): |s|>n,\\
&\text{s.t. } \Delta(s)=-1\}.
\end{aligned}
\end{equation}

The delay observer to determine the delays to be recorded is defined as an automaton
\begin{equation}\label{DO}
\mathcal{O}_{E_o}(\hat{G}) = (Y, E_o, \beta, y_{0}),
\end{equation}
where each state $ y\in 2^{\hat{X} \times \{0, 1, \dots, \lceil \frac{|a_1a_2|}{T} \rceil, \infty\}} $ represents the state estimation with delay value.
The transition function $ \beta: Y \times E_o \to Y $ and the initial state $ y_{0} $ are built as in Algorithm \ref{AOS}.
Algorithm \ref{AOS} collects the initial states with the delay value $ \infty $ into $y_0 $, and uses two iterative procedures to build the delay observer.
The procedure \texttt{RECORD1}($ \hat{x}, u, y $) collects the state estimation for $ y $ while appending the delay value to each state in the state estimation and collecting them in $ y $.
The procedure \texttt{RECORD2}($ y $) explores new transitions and observer states while appending the delay value to the ``initial'' states in the state estimation.

\makeatletter
\def\BState{\State\hskip-\ALG@thistlm}
\makeatother

\begin{algorithm} [htbp]
	\begin{algorithmic}[1]
		\caption{Construction of the delay-observer $ \mathcal{O}_{E_o}(\hat{G}) $
		} \label{AOS}
		\Require
		$ \hat{G} = (\hat{X}, E, \hat{\alpha}, \hat{X}_{0}), \lceil \frac{D}{T} \rceil, E_o $, $ \hat{X}^F, \hat{X}^H $;
		\Ensure
		$ \mathcal{O}_{E_o}(\hat{G}) = (Y, E_o, \beta, y_{0}) $;
		
		\State $ y_0 \gets \emptyset $;
		\For {$ \hat{x}_0 \in \hat{X}_0 $}
		\State $ y_0 \gets y_0 \cup (\hat{x}_0,\infty) $; \texttt{RECORD1}($\hat{x}_0,\infty,y_0$);
		\EndFor
		\State  $ Y \gets \{y_0\} $; 
		\texttt{RECORD2}$ (y_{0}) $; 
		\textbf{Return} $ \mathcal{O}_{E_o}(\hat{G}) $;
		\vspace{0.6ex}
		\Procedure{\texttt{RECORD1}}{$ \hat{x}, u, y $}
		\For {$ e\in E_{uo} :\hat{\alpha}(\hat{x},e)! $}
		\For{$ \hat{x}'\in \hat{\alpha}(\hat{x},e) $}
		\If {$ 0<u<\infty $}
		\State $ u'\gets u-1 $;
		\Else
		\State $ u'\gets u $;
		\EndIf
		\If {$ (\hat{x}',u')\notin y $}
		\State $ y\gets y\cup \{(\hat{x}',u')\} $; 
		\State \texttt{RECORD1}$ (\hat{x}', u', y) $;
		\EndIf
		\EndFor
		\EndFor
		\EndProcedure	
		\vspace{0.6ex}
		\Procedure{\texttt{RECORD2}}{$ y $}
		\State $ \iota\gets\{\hat{x}\in\hat{X}\mid(\hat{x},u)\in y\} $; $ E^{\rm Rec}\gets \emptyset $;
		\For {$ e \in E_o $}
		\State $ \iota^*\gets \{\hat{x}\in\hat{X}\mid \exists \hat{x}'\in\iota, \ \text{s.t.}\ \hat{x}\in\hat{\alpha}(\hat{x}',e) \} $;
		\If {$ \iota^*\cap \hat{X}^H=\emptyset $}
		\State $ E^{\rm Rec}\gets E^{\rm Rec}\cup\{e\} $;
		\EndIf
		\State $ y^e\gets \emptyset $;
		\EndFor
		\For {$ (\hat{x},u)\in y $}
		\For {$ e\in E_o:\hat{\alpha}(\hat{x},e)! $}
		\If {$ u=\infty \wedge e \in E^{\rm Rec} $}
		\State $ u'\gets \lceil \frac{D}{T} \rceil $;
		\ElsIf {$ 0<u<\infty $}
		\State $ u'\gets u-1 $;
		\Else
		\State $ u'\gets u $;
		\EndIf
		\For{$ \hat{x}'\in \hat{\alpha}(\hat{x},e) $}
		\If {$ (\hat{x}',u')\notin y^e $}
		\State $ y^e \gets y^e \cup \{(\hat{x}',u')\} $;
		\State \texttt{RECORD1}$ (\hat{x}', u', y^e) $;
		\EndIf
		\EndFor
		\EndFor
		\EndFor
		\For {$ e \in E_o $}
		\State Add $ \beta(y,e)=y^e $ to $ \mathcal{O}_{E_o}(\hat{G}) $;
		\If {$ y^e \notin Y $}
		\State $ Y \gets Y \cup \{y^e\} $; 
		\texttt{RECORD2}$ (y^e) $;
		\EndIf
		\EndFor
		\EndProcedure
	\end{algorithmic}
\end{algorithm}

The delay observer has the same structure as a classical observer \cite{cassandras2009introduction}, where the only difference lies in the states with the assigned delay value.
Given any $ y\in Y $, any pair of two states $ (\hat{x},u), (\hat{x}',u') \in y $ satisfies that there exist $ s,s'\in\mathcal{L}(G): \hat{x}\in\hat{\alpha}(s), \hat{x}'\in\hat{\alpha}(s') $ such that $ P_{E_o}(s)=P_{E_o}(s') $.
For each $ s\in\mathcal{L}(G): \hat{\alpha}(s)\subseteq\hat{X}^F $, $\mathcal{O}_{E_o}(\hat{G})$ records the delay of the events in $ E_o $ that help to distinguish $ s $ from the system trajectories $ s'\in\mathcal{L}(G): \hat{\alpha}(s')\cap\hat{X}^H\neq\emptyset $.
This is needed to verify $ K^T $-diagnosability and $ T $-constrained prognosability, as it will be clear in the next section.

\begin{example}\label{EDO}(Building the delay observer).
	For the system $ G $ in Fig. \ref{System}, we consider $ D=3 $, $ K=3 $ and $ T=2 $, and build the corresponding fault automaton $ \hat{G} $ as shown in Fig. \ref{FFA}.
	Given $ \lceil \frac{D}{T} \rceil=2 $, $ E_o=\{e_1,e_2,e_3\} $ and $ \hat{X}^F, \hat{X}^H $ in \eqref{diag}, we run Algorithm \ref{AOS} to obtain the delay observer $ \mathcal{O}_{E_o}^{\rm diag}(\hat{G}) $ shown in Fig. \ref{FObsdiag}, where one can see that only the delay of $ \mathbf{e_1} $ (denoted with bold) in $ e_3uf\mathbf{e_1}e_2e_2 $ is recorded.
	The delay value ``$ 0 $" in $ (6,3,0) $ indicates that the occurrence of $ \mathbf{e_1} $ that helps to distinguish $ e_3uf\mathbf{e_1}e_2e_2 $ from other system trajectories $ s'\in\mathcal{L}(G): \hat{\alpha}(s')\cap\hat{X}^H\neq\emptyset $ has been received, which implies that the fault can be diagnosed, as shown in Example \ref{EKTD}.
	On the other hand, given $ \hat{X}^F, \hat{X}^H $ in \eqref{prog}, we obtain $ \mathcal{O}_{E_o}^{\rm diag}(\hat{G}) $ shown in Fig. \ref{FObsprog} with Algorithm \ref{AOS}, where the delay of $ \mathbf{e_3} $ in $ \mathbf{e_3}ufe_1e_2e_2 $ is recorded.
	\hfill \ensuremath{\Box}
\end{example}

\begin{figure}[htbp]
	\centering
	\subfigure[$ \hat{G} $]
	{\label{FFA}
		\includegraphics[height=1.7cm]{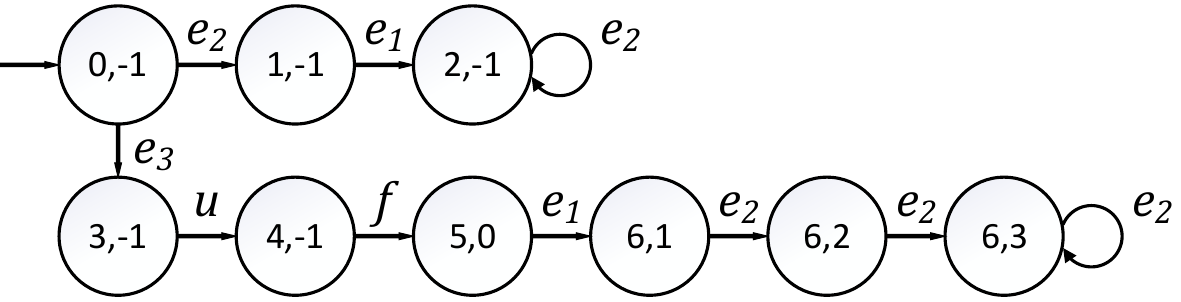}}
	\subfigure[$ \mathcal{O}_{E_o}^{\rm diag}(\hat{G}) $]
	{\label{FObsdiag}
		\includegraphics[height=1.7cm]{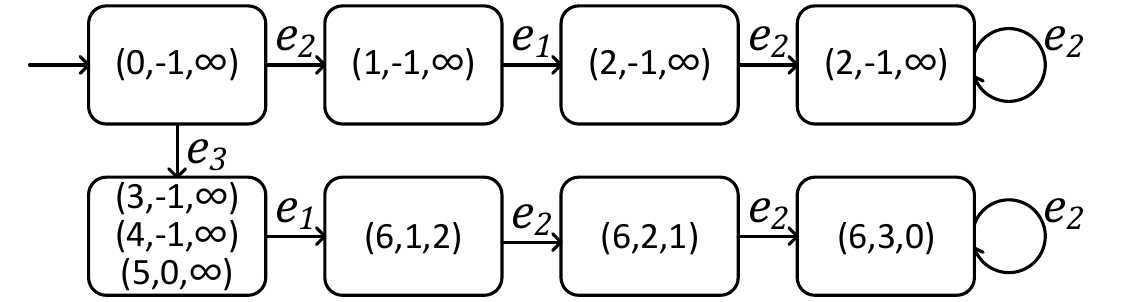}}
	\subfigure[$ \mathcal{O}_{E_o}^{\rm prog}(\hat{G}) $]
	{\label{FObsprog}
		\includegraphics[height=1.7cm]{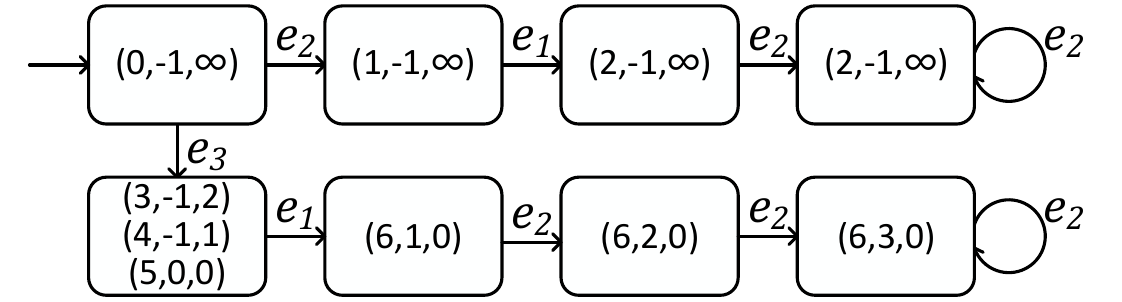}}
	\caption{Fault automaton $ \hat{G} $, and the delay observers $ \mathcal{O}_{E_o}^{\rm diag}(\hat{G}) $ and $ \mathcal{O}_{E_o}^{\rm prog}(\hat{G}) $ for fault diagnosis and prognosis with $ \lceil \frac{D}{T} \rceil=2 $.}
\end{figure}

%
%

\subsection{Verification of $ K^T\! $-diagnosability and $T$-constrained prognosability}\label{SV}

Using all structures above, we initiate the verification process of $ K^T $-diagnosability and $T$-constrained prognosability.

In Section \ref{SDO}, we embed all delays of necessary events into the delay observer, in which all states will be checked by the following verification function.
Define the mapping $ \psi : y \rightarrow \{H,F\} $ as follows:
$\forall y \in Y, $
\begin{equation}\label{vf}
\psi(y) = \left\{
\begin{aligned}
F , & \quad \text{if } y \in \mathcal{VC};\\
H , & \quad \text{otherwise};
\end{aligned} 
\right.
\end{equation}
where the condition $ y \in \mathcal{VC} $ is defined as:
\begin{equation}\label{vc}
\exists (\hat{x},u)\in y, \text{ s.t. } \hat{x}\in \hat{X}^{F}, u > 0.
\end{equation}
Now we are in the position to verify $ K^T $-diagnosability and $T$-constrained prognosability with the following theorem.

\begin{theorem}\label{TKT}
	Let $G$ in (\ref{S}) be the system model, $f$ be the fault events, $ D $ be the transmission distance, $ \mathcal{O}_{E_o}(\hat{G}) $ be the delay observer built with Algorithm \ref{AOS}.
	Then, if $ \hat{X}^F $ and $ \hat{X}^H $ are as in \eqref{diag}, $ \mathcal{L}(G) $ is $ K^{T} $-diagnosable w.r.t. $ f $ and $ D $ if and only if 
	\begin{equation}\label{SI}
	\forall y \in Y, \psi(y) = H.
	\end{equation}
	If $ \hat{X}^F $ and $ \hat{X}^H $ are as in \eqref{prog}, $ \mathcal{L}(G) $ is $T$-constrained prognosable w.r.t. $ f $ and $ D $ if and only if \eqref{SI} holds.
\end{theorem}

\begin{proof}
	Let us only prove the $ K^T $-diagnosability of the theorem, as the proof for $T$-constrained prognosability follows similar lines.
	
	(\underline{Necessity}):
	Seeking a contradiction argument, let us first suppose that $ \mathcal{L}(G) $ is $ K^{T} $-diagnosable w.r.t. $f$ and $ D $ while $ \exists y \in Y $, s.t. $ \psi(y) = F $. 
	Then, we have that $ y $ satisfies (\ref{vc}).	
	
	We now consider event strings $ s\in \mathcal{L}(\hat{G}): \hat{x}^k \in \hat{\alpha}(s) $ and $ s'\in \bar{s}: |s|-|s'|=\lceil \frac{D}{T} \rceil $, where $ \hat{x} $ is in \eqref{vc}.
	Then, we have that $ s'\in \zeta_{E_o}^{\lceil \frac{D}{T} \rceil}(s) $.
	Since $ u > 0 $, we can obtain from the procedures \texttt{RECORD1} and \texttt{RECORD2} (cf. lines 25 and 32) in Algorithm \ref{AOS} that there exists $ (\hat{x}'',\infty) \in \beta(P_{E_o}(s')) $ such that $ |\hat{x}''|_f=-1 $;
	otherwise, we have that $ \forall (\hat{x}',u')\in \beta(P_{E_o}(s')),  |\hat{x}'|_f\geq 0 $, which is impossible since the delay $ \infty $ will never appear in such an observer state according to the procedure \texttt{RECORD2}. 
	Then, we have that there exists $ s''\in \mathcal{L}(\hat{G}): \hat{x}''\in \hat{\alpha}(s'') $ such that $ P_{E_o}(s'')=P_{E_o}(s') \in P_{E_o}(\zeta_{E_o}^{\lceil \frac{D}{T} \rceil}(s'')) \cap P_{E_o}(\zeta_{E_o}^{\lceil \frac{D}{T} \rceil}(s)) $.
	Finally, with $ \Delta(s)=K $ and $ \Delta(s'')=-1 $, we conclude that $ \mathcal{L}(G) $ is not $ K^{T} $-diagnosable w.r.t. $f$ and $ D $, resulting in a contradiction.
	
	(\underline{Sufficiency}):
	Seeking a contradiction argument, let us suppose that $ \forall y \in Y, \psi(z) = H $ while $ \mathcal{L}(G) $ is not $ K^{T} $-diagnosable w.r.t. $f$ and $ D $. 
	Then, we have that there exists $ s,s' \in \mathcal{L}(G): \Delta(s) = K, \Delta(s') = -1 $ such that	$ P_{E_o}(\zeta_{E_o}^{\lceil \frac{D}{T} \rceil}(s)) \cap P_{E_o}(\zeta_{E_o}^{\lceil \frac{D}{T} \rceil}(s')) \neq \emptyset $.
	
	We first consider the case that there exists $ s''\in \mathcal{L}(\hat{G}): P_{E_o}(s'') = P_{E_o}(s) $ such that $ \Delta(s'')=-1 $.
	According to the procedures \texttt{RECORD1} and \texttt{RECORD2} (cf. lines 25 and 32) in Algorithm \ref{AOS}, there exists $ (\hat{x},u)\in \beta(P_{E_o}(s)) $ such that $ |\hat{x}|_f=K $ and $ u=\infty $, implying $ \psi(\beta(P_{E_o}(s))) = F $, resulting in a contradiction.
	
	We then consider the case that $ \forall s''\in \mathcal{L}(\hat{G}): P_{E_o}(s'') = P_{E_o}(s),  \Delta(s'')\neq-1 $.
	In this case, there must exist $ s^k\in \bar{s}, e \in E_o : s^ke\in \bar{s} $ such that $ \exists s^{-1}\in \bar{s}' $, s.t. $ P_{E_o}(s^{-1})=P_{E_o}(s^k) \in P_{E_o}(\zeta_{E_o}^{\lceil \frac{D}{T} \rceil}(s')) \cap P_{E_o}(\zeta_{E_o}^{\lceil \frac{D}{T} \rceil}(s)) $. 
	Note that $ \Delta(s^{-1})=-1 $.
	From the definition of $ \zeta $, we can obtain that $ s^k\in \zeta_{E_o}^{\lceil \frac{D}{T} \rceil}(s) $, which implies that $ |s|-|s^ke|<\lceil \frac{D}{T} \rceil $ since $ e\in\omega(s^k) $ and $ |s|-|s^ke|\geq \lceil \frac{D}{T} \rceil $ would lead to $ s^k\notin \zeta_{E_o}^{\lceil \frac{D}{T} \rceil}(s) $.
	According to the procedures \texttt{RECORD1} and \texttt{RECORD2} (cf. lines 25 and 32) in Algorithm \ref{AOS}, there exists $ (\hat{x}_1,u_1)\in \beta(P_{E_o}(s^ke)) $ such that $ u_1=\infty $ or $ \lceil \frac{D}{T} \rceil $.
	Since $ |s|-|s^ke|<\lceil \frac{D}{T} \rceil $, there exists $ (\hat{x}_2,u_2)\in \beta(P_{E_o}(s)) $ such that $ |\hat{x}_2|_f=K $ and $ u_2>u_1-\lceil \frac{D}{T} \rceil \geq 0 $, implying $ \psi(\beta(P_{E_o}(s))) = F $, resulting in a contradiction, which completes the proof.
\end{proof}

\begin{remark}
	(Complexity analysis).
	The computational complexity of the proposed verification methods primarily arises from the construction of the delay observer, which has a complexity similar to that of the observer \cite{cassandras2009introduction} in the literature.
	Indeed, the only difference between these two observers lies in the addition of an extra delay value in the states.
\end{remark}

\begin{example}\label{EKT}(Verifying $ K^T $-diagnosability and $T$-constrained prognosability).
	Consider the delay observer in Fig. \ref{FObsdiag} and \ref{FObsprog}.
	We use \eqref{vf} to check the states in $ \mathcal{O}_{E_o}^{\rm diag}(\hat{G}) $ and $ \mathcal{O}_{E_o}^{\rm prog}(\hat{G}) $.
	We know that in diagnosis problem, $ \hat{X}^F=\{\hat{x}\in \hat{X}\mid |\hat{x}|_f=3\} $ while in prognosis problem, $ \hat{X}^F=\{\hat{x}\in \hat{X}\mid |\hat{x}|_f=0\} $.
	We have $ \psi(\{(6,3,0)\})=H $ in $ \mathcal{O}_{E_o}^{\rm diag}(\hat{G}) $ and $ \psi(\{(3,-1,2),(4,-1,1),(5,0,0)\})=H $ in $ \mathcal{O}_{E_o}^{\rm prog}(\hat{G}) $, indicating that $ \mathcal{L}(G) $ is $ K^{T} $-diagnosable and $ T $-constrained prognosable w.r.t. $f$ and $ D $ for $ K=3 $, $ T=2 $, $ D=3 $ as shown in Example \ref{EKTD} and \ref{ETDP}.
	\hfill \ensuremath{\Box}
\end{example}

\section{Conclusion}\label{8}

This work extended and solved fault diagnosis and prognosis in discrete event systems in the delayed observation scenario.
In order to make this extension possible, we proposed a new diagnosability notion, namely $ K^T $-diagnosability, as well as a new prognosability notion, namely $T$-constrained prognosability, extended from a restated prognosability definition to avoid conservativeness.
To handle the delayed observations, a delay observer was constructed, that allows to verify the aforementioned conditions of diagnosability and prognosability.
Interesting directions for future work are to consider multiple delays, dynamic observations and stochastic settings.

\bibliographystyle{IEEEtran}
\bibliography{References}

@book{cassandras2009introduction,
  title={Introduction to discrete event systems},
  author={Cassandras, Christos G and Lafortune, Stephane},
  year={2009},
  publisher={Springer Science \& Business Media}
}

@article{sampath1995diagnosability,
  title={Diagnosability of discrete-event systems},
  author={Sampath, Meera and Sengupta, Raja and Lafortune, St{\'e}phane and Sinnamohideen, Kasim and Teneketzis, Demosthenis},
  journal={IEEE Transactions on Automatic Control},
  volume={40},
  number={9},
  pages={1555--1575},
  year={1995},
  publisher={IEEE}
}

@article{cassez2008fault,
  title={Fault diagnosis with static and dynamic observers},
  author={Cassez, Franck and Tripakis, Stavros},
  journal={Fundamenta Informaticae},
  volume={88},
  number={4},
  pages={497--540},
  year={2008},
  publisher={IOS Press}
}

@article{yin2015codiagnosability,
  title={Codiagnosability and coobservability under dynamic observations: Transformation and verification},
  author={Yin, Xiang and Lafortune, St{\'e}phane},
  journal={Automatica},
  volume={61},
  pages={241--252},
  year={2015},
  publisher={Elsevier}
}

@article{debouk2000coordinated,
  title={Coordinated decentralized protocols for failure diagnosis of discrete event systems},
  author={Debouk, Rami and Lafortune, St{\'e}phane and Teneketzis, Demosthenis},
  journal={Discrete Event Dynamic Systems},
  volume={10},
  number={1-2},
  pages={33--86},
  year={2000},
  publisher={Springer}
}

@article{yin2019robust,
  title={Robust fault diagnosis of stochastic discrete event systems},
  author={Yin, Xiang and Chen, Jun and Li, Zhaojian and Li, Shaoyuan},
  journal={IEEE Transactions on Automatic Control},
  volume={64},
  number={10},
  pages={4237--4244},
  year={2019},
  publisher={IEEE}
}

@article{qiu2008distributed,
  title={Distributed diagnosis under bounded-delay communication of immediately forwarded local observations},
  author={Qiu, Wenbin and Kumar, Ratnesh},
  journal={IEEE Transactions on Systems, Man, and Cybernetics-Part A: Systems and Humans},
  volume={38},
  number={3},
  pages={628--643},
  year={2008},
  publisher={IEEE}
}

@article{takai2011distributed,
  title={Distributed failure prognosis of discrete event systems with bounded-delay communications},
  author={Takai, Shigemasa and Kumar, Ratnesh},
  journal={IEEE Transactions on Automatic Control},
  volume={57},
  number={5},
  pages={1259--1265},
  year={2011},
  publisher={IEEE}
}

@article{lafortune2018history,
  title={On the history of diagnosability and opacity in discrete event systems},
  author={Lafortune, St{\'e}phane and Lin, Feng and Hadjicostis, Christoforos N},
  journal={Annual Reviews in Control},
  volume={45},
  pages={257--266},
  year={2018},
  publisher={Elsevier}
}

@article{lin1994diagnosability,
  title={Diagnosability of discrete event systems and its applications},
  author={Lin, Feng},
  journal={Discrete Event Dynamic Systems},
  volume={4},
  number={2},
  pages={197--212},
  year={1994},
  publisher={Springer}
}

@article{nunes2018codiagnosability,
  title={Codiagnosability of networked discrete event systems subject to communication delays and intermittent loss of observation},
  author={Nunes, Carlos EV and Moreira, Marcos V and Alves, Marcos VS and Carvalho, Lilian K and Basilio, Jo{\~a}o Carlos},
  journal={Discrete Event Dynamic Systems},
  volume={28},
  number={2},
  pages={215--246},
  year={2018},
  publisher={Springer}
}

@article{genc2009predictability,
  title={Predictability of event occurrences in partially-observed discrete-event systems},
  author={Genc, Sahika and Lafortune, St{\'e}phane},
  journal={Automatica},
  volume={45},
  number={2},
  pages={301--311},
  year={2009},
  publisher={Elsevier}
}

@article{watanabe2021fault,
  title={Fault prognosis of discrete event systems: An overview},
  author={Watanabe, Ana TY and Sebem, Renan and Leal, Andre B and Hounsell, Marcelo da S},
  journal={Annual Reviews in Control},
  volume={51},
  pages={100--110},
  year={2021},
  publisher={Elsevier}
}

@article{kumar2009decentralized,
  title={Decentralized prognosis of failures in discrete event systems},
  author={Kumar, Ratnesh and Takai, Shigemasa},
  journal={IEEE Transactions on Automatic Control},
  volume={55},
  number={1},
  pages={48--59},
  year={2009},
  publisher={IEEE}
}

@article{liu2009diagnosability,
  title={Diagnosability of fuzzy discrete-event systems: A fuzzy approach},
  author={Liu, Fuchun and Qiu, Daowen},
  journal={IEEE Transactions on Fuzzy Systems},
  volume={17},
  number={2},
  pages={372--384},
  year={2009},
  publisher={IEEE}
}

@article{benmessahel2017predictability,
  title={Predictability of fuzzy discrete event systems},
  author={Benmessahel, Bilal and Touahria, Mohamed and Nouioua, Farid},
  journal={Discrete Event Dynamic Systems},
  volume={27},
  pages={641--673},
  year={2017},
  publisher={Springer}
}

@article{paoli2011active,
  title={Active fault tolerant control of discrete event systems using online diagnostics},
  author={Paoli, Andrea and Sartini, Matteo and Lafortune, St{\'e}phane},
  journal={Automatica},
  volume={47},
  number={4},
  pages={639--649},
  year={2011},
  publisher={Elsevier}
}

@article{watanabe2017safe,
  title={Safe controllability using online prognosis},
  author={Watanabe, Ana TY and Leal, Andr{\'e} B and Cury, Jos{\'e} ER and de Queiroz, Max H},
  journal={IFAC-PapersOnLine},
  volume={50},
  number={1},
  pages={12359--12365},
  year={2017},
  publisher={Elsevier}
}

@article{he2024asynchronous,
  title={Asynchronous Fault Diagnosis of Stochastic Discrete-Event Systems in Industrial Applications},
  author={He, Jiahan and Wang, Deguang and Yang, Ming and Hu, Yuhong},
  journal={IEEE Sensors Journal},
  year={2024},
  publisher={IEEE}
}

@article{liao2022robust,
  title={Robust predictability of stochastic discrete-event systems and a polynomial-time verification},
  author={Liao, Hui and Liu, Fuchun and Wu, Naiqi},
  journal={Automatica},
  volume={144},
  pages={110477},
  year={2022},
  publisher={Elsevier}
}

@article{wang2024distributed,
  title={Distributed Fault Diagnosis in Discrete Event Systems with Transmission Delay Impairments},
  author={Wang, Jiwei and Baldi, Simone and Yu, Wenwu and Yin, Xiang},
  journal={IEEE Transactions on Automatic Control},
  year={2024},
  publisher={IEEE}
}

@article{liu2021online,
  title={Online supervisory control of networked discrete event systems with control delays},
  author={Liu, Zhaocong and Yin, Xiang and Shu, Shaolong and Lin, Feng and Li, Shaoyuan},
  journal={IEEE Transactions on Automatic Control},
  volume={67},
  number={5},
  pages={2314--2329},
  year={2021},
  publisher={IEEE}
}

@article{viana2021codiagnosability,
  title={Codiagnosability of networked discrete event systems with timing structure},
  author={Viana, Gustavo S and Alves, Marcos VS and Basilio, Jo{\~a}o Carlos},
  journal={IEEE Transactions on Automatic Control},
  volume={67},
  number={8},
  pages={3933--3948},
  year={2021},
  publisher={IEEE}
}

\end{document}